\documentclass[11pt]{amsart}
\usepackage{tikz,amsthm,amsmath,amstext,amssymb,amscd,epsfig,euscript, mathrsfs, dsfont,pspicture,multicol,graphpap,graphics,graphicx,times,enumerate,subfig,sidecap,wrapfig,color}
\usepackage{amsmath}
\usepackage{amssymb}

\usepackage{pgf}

	\definecolor{tropicalrainforest}{rgb}{0.0, 0.46, 0.37}

\newcommand{\M}{{\mathcal M}}       %
\newcommand{\R}{{\mathbb R}}       
\newcommand{\N}{{\mathbb N}}       %
\newcommand{\DD}{{\mathcal D}}
\newcommand{\CC}{{\mathcal C}}

\newcommand{\HH}{{\mathcal H}}

\newcommand{\RR}{{\mathcal R}}

\newcommand{\diam}{\mathop{\rm diam}}
\newcommand{\dist}{{\rm dist}}

\newcommand{\fiproof}{{\hspace*{\fill} $\square$ \vspace{2pt}}}

\newcommand{\rf}[1]{{(\ref{#1})}}

\newcommand{\supp}{\operatorname{supp}}
\newcommand{\CCap}{\operatorname{Cap}}

\newcommand{\vphi}{{\varphi}}
\newcommand{\ve}{{\varepsilon}}
\newcommand{\vv}{{\vspace{2mm}}}
\newcommand{\vvv}{\vspace{4mm}}
\newcommand{\wt}[1]{{\widetilde{#1}}}

\newcommand{\bad}{{\mathsf{Bad}}}
\newcommand{\good}{{\mathsf{Good}}}

\newcommand{\HD}{{\mathsf{HD}}}
\newcommand{\LM}{{\mathsf{LM}}}

\def\Xint#1{\mathchoice
{\XXint\displaystyle\textstyle{#1}}%
{\XXint\textstyle\scriptstyle{#1}}%
{\XXint\scriptstyle\scriptscriptstyle{#1}}%
{\XXint\scriptscriptstyle\scriptscriptstyle{#1}}%
\!\int}
\def\XXint#1#2#3{{\setbox0=\hbox{$#1{#2#3}{\int}$ }
\vcenter{\hbox{$#2#3$ }}\kern-.58\wd0}}

\def\avint{\Xint-}

\newtheorem{theorem}{Theorem}[section]
\newtheorem{lemma}[theorem]{Lemma}

\newtheorem{propo}[theorem]{Proposition}

\newtheorem*{lemma*}{Lemma}
\newtheorem*{theorem*}{Theorem}

\theoremstyle{definition}

\theoremstyle{remark}
\newtheorem{rem}[theorem]{\bf Remark}

\numberwithin{equation}{section}

\newcommand{\brem}{\begin{rem}}
\newcommand{\erem}{\end{rem}}

\newcommand{\re}{{\mathbb R}}
\newcommand{\pom}{{\partial\Omega}}

\begin{document}

\title[Harmonic measure and rectifiability]{Absolute continuity between the surface measure and harmonic measure implies rectifiability}

\author[S. Hofmann]{Steve Hofmann}

\address{Steve Hofmann
\\
Department of Mathematics
\\
University of Missouri
\\
Columbia, MO 65211, USA} \email{hofmanns@missouri.edu}

\author[J.M. Martell]{Jos\'e Mar{\'\i}a Martell}

\address{Jos\'e Mar{\'\i}a Martell
\\
Instituto de Ciencias Matem\'aticas CSIC-UAM-UC3M-UCM
\\
Consejo Superior de Investigaciones Cient{\'\i}ficas
\\
C/ Nicol\'as Cabrera, 13-15
\\
E-28049 Madrid, Spain} \email{chema.martell@icmat.es}

\author[S. Mayboroda]{Svitlana Mayboroda}

\address{Svitlana Mayboroda
\\
Department of Mathematics
\\
University of Minnesota
\\
Minneapolis, MN 55455, USA} \email{svitlana@math.umn.edu}

\author[X. Tolsa]{Xavier Tolsa}
\address{Xavier Tolsa
\\
ICREA and Departament de Matem\`atiques
\\
Universitat Aut\`onoma de Barcelona
\\
Edifici C Facultat de Ci\`encies
\\
08193 Bellaterra (Barcelona), Catalonia
}
\email{xtolsa@mat.uab.cat}

\author[A. Volberg]{Alexander Volberg}
\address{Alexander Volberg
\\
Department of Mathematics
\\
Michigan State University
\\
East Lansing, MI 48824, USA}

\email{volberg@math.msu.edu}


\thanks{The first author was partially supported by NSF grant DMS 1361701. The second author has been supported in part by ICMAT Severo Ochoa project SEV-2011-0087 and he acknowledges that the research leading to these results has received funding from the European Research Council under the European Union's Seventh Framework Programme (FP7/2007-2013)/ ERC agreement no. 615112 HAPDEGMT. The third author is supported in part by the Alfred P. Sloan Fellowship, the NSF INSPIRE Award DMS 1344235, NSF CAREER Award DMS 1220089 and  NSF UMN MRSEC Seed grant DMR 0212302. The fourth author was supported by the ERC grant 320501 of the European Research Council (FP7/2007-2013),  by 2014-SGR-75 (Catalonia), MTM2013-44304-P (Spain), and by the Marie Curie ITN MAnET (FP7-607647). The last author
was partially supported by the NSF grant DMS-1265549.
The results of this paper were obtained while the authors were participating in the \textit{Research in Paris} program at the \textit{Institut Henri Poincar\'e}. All authors would like to express their gratitude to this institution for the support and the nice working environment.}

\begin{abstract}
In the present paper we prove that for any open connected set $\Omega\subset\R^{n+1}$, $n\geq 1$, and any $E\subset \pom$ with $0<\HH^n(E)<\infty$ absolute continuity of the harmonic measure $\omega$ with respect to the Hausdorff measure on $E$ implies that $\omega|_E$ is rectifiable.
\end{abstract}

\maketitle

\tableofcontents


\section{Introduction}

In \cite{AMT} the authors showed that absolute continuity of the harmonic measure with respect to the Hausdorff measure on $E$ implies rectifiability of $\omega|_E$ under an additional assumption that $\Omega$ is {\it porous} in a neighborhood of $E$. In the present work we remove the aforementioned porosity assumption. This manuscript will be combined with \cite{AMT} for publication, and for that reason we only present a shortened version of the introduction highlighting the statements of the final results.

Our main result is the following.

\begin{theorem}\label{teo1}
Let $n\geq 1$ and $\Omega\subsetneq\R^{n+1}$ be an open  connected set
and let $\omega:=\omega^p$ be the harmonic measure in $\Omega$ where $p$ is a fixed point in $\Omega$.
Suppose that there exists $E\subset\partial\Omega$ with $0<\HH^n(E)<\infty$ and that the harmonic measure $\omega|_E$ is absolutely continuous with respect to $\HH^n|_{E}$. Then
 $\omega|_E$ is $n$-rectifiable, in the sense that  $\omega$-almost all of $E$ can be covered by a countable union of $n$-dimensional (possibly rotated) Lipschitz graphs.
\end{theorem}
\vv

Deep connections between absolute continuity of the harmonic measure and rectifiability of the underlying set have for a long time been a subject of thorough investigation.
In 1916 F. and M. Riesz proved that for a simply connected domain in the complex plane, with a rectifiable boundary, harmonic measure is absolutely continuous with respect to arclength measure on
the boundary \cite{RR}. More generally, if only a portion of the boundary is
rectifiable, Bishop and Jones \cite{BJ} have shown that harmonic measure is absolutely
continuous with respect to arclength on that portion. They also demonstrate that the result of \cite{RR} may fail in the absence of some topological hypothesis (e.g., simple connectedness).

The higher dimensional analogues of \cite{BJ} include absolute continuity of harmonic measure with respect to the Hausdorff measure for Lipschitz graph domains \cite{Da} and non-tangentially accessible (NTA) domains \cite{DJ}, \cite{Se}. To be precise, \cite{Da}, \cite{DJ}, \cite{Se} establish a quantitative scale-invariant result, $A^\infty$ property of harmonic measure, which in the planar case was proved by Lavrent'ev \cite{Lv}. We shall not give a precise definition of NTA domains here, but let us mention that they necessarily satisfy interior and exterior cork-screw condition as well as Harnack chain condition, that is, certain quantitative analogues of connectivity and openness, respectively. Similarly to the lower-dimensional case, the counterexamples show that some topological restrictions are needed for absolute continuity of $\omega$ with respect to $\HH^n$ \cite{Wu}, \cite{Z}.

In the present paper we attack the converse direction, in the spirit of free boundary problems. We establish that rectifiability is {\it necessary} for absolute continuity of the harmonic measure. The main two antecedents of our work are \cite{AMT} and \cite{HMIV}. As mentioned above, in  \cite{AMT} the authors prove that absolute continuity of the harmonic measure with respect to $\HH^n$ implies rectifiability under a background hypotheses that the domain $\Omega$ is {\it porous} near $E\subset \pom$, that is, there is $r_0>0$ so that every ball $B$ centered at $E$ of radius at most $r_0$ contains another ball $B'\subset \R^{n+1}\setminus\pom$ with $r(B)\approx r(B')$, with the implicit constant depending only on $E$. In \cite{HMIV} the authors establish a quantitative analogue of this result, connecting weak-$A^\infty$ property of the harmonic measure to uniform rectifiability of the boundary of the non-necessarily connected domain. However, their background conditions (Ahlfors-David regularity of $\partial\Omega$) naturally include porosity as well.

The main achievement of the present work lies in removing the porosity assumption and establishing rectifiability of $\omega|_E$ with no a priori requirements on topological structure of the set.

We note that in Theorem~\ref{teo1} connectivity is just a cosmetic assumption needed to make sense of harmonic measure at a given pole. In the presence of multiple components, one can work with one component at a time.

We also remark that in the course of the proof of our main result we may assume that $\Omega$ is bounded.  Otherwise, we take any open ball $B$ so that $2\,B\subset \Omega$ then consider $\widetilde{\Omega}=\Omega\setminus \overline{B}$  and then the two harmonic measures (the one for $\Omega$ and the one for $\widetilde{\Omega}$) are mutually absolutely continuous on $\partial\Omega$. Then by Kelvin transform with respect to the center of the ball we can reduce matters to the case of a bounded domain. Further details are left to the interested reader.

To wrap up the discussion of the background, let us mention that a crucial ingredient of our argument, as well as that of \cite{AMT}, is the recent resolution of the David-Semmes conjecture in \cite{NToV1}, \cite{NToV2}. According to the latter, boundedness of the Riesz transforms implies rectifiability of the underlying set, and the core of the present work lies in some intricate estimates on the harmonic measure and the Green function which ultimately yield desired bounds on the Riesz transform.

Finally, we remind the reader that this paper will be combined with \cite{AMT} and a more detailed historical context will be discussed in the combined manuscript.


\vv
\section{Preliminaries}

Given $A\subset \R^{n+1}$, we denote $n$-dimensional Hausdorff measure by $\HH^n(A)$, and its $n$-dimensional Hausdorff
content by $\HH_\infty^n(A)$.

Given a signed Radon measure $\nu$ in $\R^{n+1}$ we consider the $n$-dimensional Riesz
transform
$$\RR\nu(x) = \int \frac{x-y}{|x-y|^{n+1}}\,d\nu(y),$$
whenever the integral makes sense. For $\ve>0$, its $\ve$-truncated version is given by
$$\RR_\ve \nu(x) = \int_{|x-y|>\ve} \frac{x-y}{|x-y|^{n+1}}\,d\nu(y).$$
We also consider the maximal operators
$$\M^n\nu(x) = \sup_{r>0}\frac{|\nu|(B(x,r))}{r^n},$$
and for $\ve\geq0$,
$$\M^n_\ve\nu(x) = \sup_{r>\ve}\frac{|\nu|(B(x,r))}{r^n}.$$

\vv
The following is a variant of a well known estimate due to Bourgain (see \cite{B87}).

\begin{lemma}[{\cite[Lemma 4.1]{AMT}}]
\label{lembourgain}
There is $\delta_{0}>0$ depending only on $n\geq 1$ so that the following holds for $\delta\in (0,\delta_{0})$. Let $\Omega\subsetneq \R^{n+1}$ be a domain, $\xi \in \partial \Omega$, $r>0$, $B=B(\xi,r)$, and set $\rho:=\mathcal H_\infty^{s}(\partial\Omega\cap \delta B)/(\delta r)^{s}$ for some $s>n-1$. Then
\begin{equation}\label{eqbet1}
\omega_{\Omega}^{x}(B)\gtrsim_{n} \rho\quad \mbox{  for all }x\in \delta B\cap \Omega.
\end{equation}
\end{lemma}


\section{Proof of Theorem \ref{teo1} for bounded Wiener regular domains $\re^{n+1}$, $n\ge 2$}

Our goal in this section consists in proving Theorem \ref{teo1} under the additional assumption that
$\Omega$ is bounded and that the domain is Wiener regular (that is, all boundary points are Wiener regular). We will work in $\re^{n+1}$, $n\ge 2$. Let us first rescall the definition of Wiener regular points.

For $n\geq2$, the Newtonian potential of a measure $\mu$ in $\R^{n+1}$ is defined as
$$
U^\mu(x)
:= \int \frac1{|x-y|^{n-1}}\,d\mu(y).
$$
The Newtonian capacity of Borel compact set $A\subset \R^{n+1}$ is defined by
$$
\CCap(A) = \sup\{\mu(A):\,U^\mu(x)\leq 1,\ \,\forall x\in\R^{n+1}\}.
$$

Given $\Omega\subsetneq\re^{n+1}$. We say that a point $x\in\partial\Omega$ is Wiener regular for $\Omega$ (or just regular) if
$$
\int_0^1 \frac{{\rm Cap}(A(x,r, 2r)\cap \Omega^c)}{r^{n-1}}\,\frac{dr}{r} = \infty,
$$
where $A(x,r,s)$, $r<s$, stands for the open annulus $B(x,s)\setminus \overline{B(x,r)}$.
If $x$ is not regular, we say that it is irregular. We say that $\Omega$ is Wiener regular if every $x\in\pom$ is Wiener regular.

\begin{propo}\label{propreg}
Let $n\geq 2$ and $\Omega\subset\R^{n+1}$ be a bounded  open  connected set
and let $\omega:=\omega^p$ be the harmonic measure in $\Omega$ where $p$ is a fixed point in $\Omega$. Assume further that $\Omega$ is Wiener regular. Suppose that there exists $E\subset\partial\Omega$ with $0<\HH^n(E)<\infty$ and that the harmonic measure $\omega|_E$ is absolutely continuous with respect to $\HH^n|_{E}$. Then
 $\omega|_E$ is $n$-rectifiable, in the sense that  $\omega$-almost all of $E$ can be covered by a countable union of $n$-dimensional (possibly rotated) Lipschitz graphs.
\end{propo}

The proof of this lemma will follow the same lines as the proof of Theorem 1 in \cite{AMT},
where a version of our Theorem \ref{teo1} is
obtained under the additional hypothesis of ``porosity", i.e., assuming the existence of a corkscrew
point in some component of $\mathbb{R}^{n+1}\setminus \partial\Omega$, at all scales with a uniform constant.
To remove the porosity assumption, only the Key Lemma 7.1 from \cite{AMT} needs to be modified.
However, for the reader's convenience we summarize the main ingredients from the arguments of \cite{AMT}.

\subsection{Relationship between harmonic measure and the the Green function}\label{sect-Green}

In what follows, and unless otherwise stated, $\Omega\subset\R^{n+1}$, $n\ge 2$ is a bounded  open  connected set
such that $\Omega$ is Wiener regular. We write $\omega^x$ to denote harmonic measure for $\Omega$ with pole at $x\in\Omega$. Analogously, $G$ will denote the Green function for $\Omega$ which is defined as follows. Write
$\mathcal{E}(x)=c_n\,|x|^{1-n}$ for the fundamental solution for Laplace's equation in $\R^{n+1}, n\geq 2$.
We define the Green function
\begin{equation}\label{green}
G(x,y) = \mathcal{E}(x-y) - \int_{\partial\Omega} \mathcal{E}(x-z)d\omega^y(z)=: \mathcal{E}(x-y) -v_x(y),
\end{equation}
which then satisfies \cite[Definition 4.2.3]{Hel}.
We now claim that $G$ belongs to $W^{1,2}_0(\Omega)$ away from the pole.
Indeed, by the Wiener regularity of   $\partial \Omega$,
and the fact that the data $\mathcal{E}(x-\cdot)\big|_{\partial\Omega}$ is Lipschitz for fixed $x\in \Omega$,
the solution $v_x(y)$ defined above
coincides with the Lax-Milgram solution constructed in the standard way as follows
(see, e.g., \cite[p. 5]{Ke}).  Set
$F_x(y):=\mathcal{E}(x-y)\,(1-\psi((x-y)/\delta(x)))\,\psi(y/R)$ where $\psi\in \CC_c^\infty(\mathbb{R}^{n+1})$ is radial, $0\le \psi\le 1$, ${\rm \supp\,} \psi\subset B(0,1/2)$ and $\psi\equiv 1$ on $B(0,1/4)$
and where $R$ is large enough so that $\Omega\subset B(0,R/8)$.  Clearly, $F_x\in \CC_c^\infty(\mathbb{R}^{n+1})$,
and $F_x\big|_{\partial\Omega}=\mathcal{E}(x-\cdot)$.  Thus, by Lax-Milgram we may
construct $u_x\in W^{1,2}_0(\Omega)$ such that $\mathcal{L} u_x= \mathcal{L}F_x\in W^{-1,2}(\Omega)$,
where $\mathcal{L}$ is the Laplacian.  Then $v_x = F_x - u_x,$  and therefore
$$
G(x,y)
=
\mathcal{E}(x-y)-v_x(y)
=
\mathcal{E}(x-y)-F_x(y)+u_x
$$
Since $\mathcal{E}(x-\cdot)$ and $F_x(\cdot)$ agree in $\Omega\setminus B(x,\delta(x)/2) $,
it follows that $\mathcal{E}(x-\cdot)- F_x(\cdot)$ is in  $W_0^{1,2}(\Omega)$ away from $x$.
Moreover, by construction $u_x\in W^{1,2}_0(\Omega)$, so
that $G(x,\cdot)\in W^{1,2}_0(\Omega)$ away from $x$ as desired.

The following auxiliary result is  somewhat similar to \cite[Lemma 4.2 ]{AMT}. The main difference is that $y$ is required to be a corkscrew point $x_B$ relative to $B$ and also that one can replace the infimum in
\rf{eq:Green-lowerbound} just by $\omega^{x_B}(B)$.

\begin{lemma}\label{l:w>G}
Let $n\ge 2$ and $\Omega\subset\R^{n+1}$ be a bounded open connected set which  is Wiener regular.
Let $B=B(x,r)$ be a closed ball with $x\in\pom$ and $0<r<\diam(\pom)$. Then, for all $a>0$,
\begin{equation}\label{eq:Green-lowerbound}
 \omega_{\Omega}^{x}(aB)\gtrsim \inf_{z\in 2B\cap \Omega} \omega_\Omega^{z}(aB)\, r^{n-1}\, G(x,y)\quad\mbox{
 for all $x\in \Omega\backslash  2B$ and $y\in B\cap\Omega$,}
 \end{equation}
 with the implicit constant independent of $a$.
\end{lemma}

\begin{proof}
Fix $y\in B\cap\Omega$ and note that for every $x\in\partial (2B)\cap\Omega$
we have
\begin{equation}\label{eqop1}
G(x,y)\lesssim \frac1{|x-y|^{n-1}}\leq \frac c{r^{n-1}} \leq \frac{c\,\omega^x(aB)}{r^{n-1}\,\inf_{z\in 2B\cap \Omega} \omega^{z}(aB)}
.
\end{equation}
Let us observe that the two functions
$$
u(x)=c^{-1}\,G(x, y)\,r^{n-1}\,\inf_{z\in 2B\cap \Omega} \omega^{z}(aB)\qquad\text{ and }\qquad v(x)=\omega^x(aB)
$$
are harmonic in $\Omega\setminus 2B$ and \eqref{eqop1} says that $u\le v$ in $\partial (2B)\cap\Omega$.
 We would like to use maximum principle in the domain $\Omega\setminus 2B$, but in order to rigorously justify that use we need to approximate $v$. Let $\psi_\epsilon\in \CC^\infty_c(\re^{n+1})$ be a radial function, $0\le \psi_\epsilon\le 1$, $\psi_\epsilon$ supported in $(a+\epsilon)B$ and $\psi_\epsilon\equiv 1$ in $a\,B$. Note then that
$$
v(x)=
\int_{\pom} \chi_{aB}\,d\omega^x
\le
\int_{\pom} \psi_\epsilon\,d\omega^x
=:
v_\epsilon(x).
$$
Since $\psi_\epsilon$ is smooth and all boundary points are Wiener regular we can conclude that $v_\epsilon\in
\CC (\overline{\Omega})$. Hence $u$, $v_\epsilon\in\CC(\overline{\Omega\setminus 2\,B})$ (that $u$ is continuous away from the pole follows again from \eqref{green} and the Wiener regularity). We now claim that $u\le v_\epsilon$ on $\partial(\Omega\setminus 2B)$. Indeed, from what we showed before $u\le v\le v_\epsilon$ on $\partial (2B)\cap\Omega$ and also
$u(x)=0$ in $\partial\Omega$ and $v_\epsilon(x)\ge 0$ for every $x\in \overline{\Omega}$. Hence the maximal principle
for continuous solutions all the way to the boundary yields that $u\le v_\epsilon$ on $\Omega\setminus 2\,B$. To conclude with our estimate we just need to observe that $v_\epsilon(x)\to v(x)$ for every $x\in \Omega\setminus 2\,B$ by dominated convergence theorem and the fact that $\psi_\epsilon(z)\to \chi_{a\,B}(z)$ for everywhere $z\in\re^{n+1}$.
\end{proof}

\vv

\subsection{The dyadic lattice of David and Mattila}\label{secdya}

We introduce now the dyadic lattice of cubes
with small boundaries of David-Mattila associated with $\omega^p$, where $p$ is a fixed pole in $\Omega$, from  \cite[Theorem 3.2]{David-Mattila}.

\begin{lemma}[David, Mattila]
\label{lemcubs}
Consider two constants $C_0>1$ and $A_0>5000\,C_0$ and denote $W=\supp\omega^p$. Then there exists a sequence of partitions of $W$ into
Borel subsets $Q$, $Q\in \DD_k$, with the following properties:
\begin{itemize}
\item For each integer $k\geq0$, $W$ is the disjoint union of the ``cubes'' $Q$, $Q\in\DD_k$, and
if $k<l$, $Q\in\DD_l$, and $R\in\DD_k$, then either $Q\cap R=\varnothing$ or else $Q\subset R$.
\vv

\item The general position of the cubes $Q$ can be described as follows. For each $k\geq0$ and each cube $Q\in\DD_k$, there is a ball $B(Q)=B(z_Q,r(Q))$ such that
$$z_Q\in W, \qquad A_0^{-k}\leq r(Q)\leq C_0\,A_0^{-k},$$
$$W\cap B(Q)\subset Q\subset W\cap 28\,B(Q)=W \cap B(z_Q,28r(Q)),$$
and
$$\mbox{the balls\, $5B(Q)$, $Q\in\DD_k$, are disjoint.}$$

\vv
\item The cubes $Q\in\DD_k$ have small boundaries. That is, for each $Q\in\DD_k$ and each
integer $l\geq0$, set
$$N_l^{ext}(Q)= \{x\in W\setminus Q:\,\dist(x,Q)< A_0^{-k-l}\},$$
$$N_l^{int}(Q)= \{x\in Q:\,\dist(x,W\setminus Q)< A_0^{-k-l}\},$$
and
$$N_l(Q)= N_l^{ext}(Q) \cup N_l^{int}(Q).$$
Then
\begin{equation}\label{eqsmb2}
\omega^p(N_l(Q))\leq (C^{-1}C_0^{-3d-1}A_0)^{-l}\,\omega^p(90B(Q)).
\end{equation}
\vv

\item Denote by $\DD_k^{db}$ the family of cubes $Q\in\DD_k$ for which
\begin{equation}\label{eqdob22}
\omega^p(100B(Q))\leq C_0\,\omega^p(B(Q)).
\end{equation}
We have that $r(Q)=A_0^{-k}$ when $Q\in\DD_k\setminus \DD_k^{db}$
and
\begin{equation}\label{eqdob23}
\omega^p(100B(Q))\leq C_0^{-l}\,\omega^p(100^{l+1}B(Q))\quad
\mbox{for all $l\geq1$ such that $100^l\leq C_0$ and $Q\in\DD_k\setminus \DD_k^{db}$.}
\end{equation}
\end{itemize}
\end{lemma}

\vv

We use the notation $\DD=\bigcup_{k\geq0}\DD_k$. Observe that the families $\DD_k$ are only defined for $k\geq0$. So the diameter of the cubes from $\DD$ are uniformly
bounded from above. We set
$\ell(Q)= 56\,C_0\,A_0^{-k}$ and we call it the side length of $Q$. Notice that
$$\frac1{28}\,C_0^{-1}\ell(Q)\leq \diam(B(Q))\leq\ell(Q).$$
Observe that $r(Q)\sim\diam(B(Q))\sim\ell(Q)$.
Also we call $z_Q$ the center of $Q$, and the cube $Q'\in \DD_{k-1}$ such that $Q'\supset Q$ the parent of $Q$.
 We set
$B_Q=28 \,B(Q)=B(z_Q,28\,r(Q))$, so that
$$W\cap \tfrac1{28}B_Q\subset Q\subset B_Q.$$
For $Q\in\DD$, we write $J(Q)\in\N$ if $Q\in\DD_{J(Q)}$.

We denote
$\DD^{db}=\bigcup_{k\geq0}\DD_k^{db}$.
Note that, in particular, from \rf{eqdob22} it follows that
\begin{equation}\label{eqdob*}
\omega^{p}(3B_{Q})\leq \omega^p(100B(Q))\leq C_0\,\omega^p(Q)\qquad\mbox{if $Q\in\DD^{db}.$}
\end{equation}
For this reason we will call the cubes from $\DD^{db}$ doubling.

As shown in \cite[Lemma 5.28]{David-Mattila}, every cube $R\in\DD$ can be covered $\omega^p$-a.e.\
by a family of doubling cubes:
\vv

\begin{lemma}\label{lemcobdob}
Let $R\in\DD$. Suppose that the constants $A_0$ and $C_0$ in Lemma \ref{lemcubs} are
chosen suitably. Then there exists a family of
doubling cubes $\{Q_i\}_{i\in I}\subset \DD^{db}$, with
$Q_i\subset R$ for all $i$, such that their union covers $\omega^p$-almost all $R$.
\end{lemma}

Given a ball $B\subset \R^{n+1}$, we consider its $n$-dimensional density:
$$\Theta_\omega(B)= \frac{\omega^p(B)}{r(B)^n}.$$

The following is an easy consequence of \cite[Lemma 5.31]{David-Mattila}. For the precise details, see
\cite[Lemma 4.4]{Tolsa-memo}, for example.

\vv
\begin{lemma}\label{lemcad23}
Let $R\in\DD$ and let $Q\subset R$ be a cube such that all the intermediate cubes $S$,
$Q\subsetneq S\subsetneq R$ are non-doubling (i.e.\ belong to $\DD\setminus \DD^{db}$).
Then
$$\Theta_\omega(100B(Q))\leq C_0\,A_0^{-9n(J(Q)-J(R)-1)}\,\Theta_\omega(100B(R))$$
and
$$\sum_{S\in\DD:Q\subset S\subset R}\Theta_\omega(100B(S))\leq c\,\Theta_\omega(100B(R)),$$
with $c$ depending on $C_0$ and $A_0$.
\end{lemma}

\vv
From now on we will assume that $C_0$ and $A_0$ are some big fixed constants so that the
results stated in the lemmas of this section hold.

\vvv
\subsection{The Frostman measure}

From now on, $\Omega$ and $E$ will be as in Proposition \ref{propreg}.
We fix a point $p\in\Omega$ and consider the harmonic measure $\omega^p$ of $\Omega$ with pole at $p$. We may assume thet $\omega^p(E)>0$, otherwise there is nothing to prove.

Let $g\in L^1(\omega^p)$ be such that
$$\omega^p|_E = g\,\HH^n|_{\partial\Omega}.$$
Given $M>0$, let
$$E_M= \{x\in\partial\Omega:M^{-1}\leq g(x)\leq M\}.$$
Take $M$ big enough so that $\omega^p(E_M)\geq \omega^p(E)/2>0$.
Consider an arbitrary compact set $F_M\subset E_M$ with $\omega^p(F_M)>0$.

Let $\mu$ be an $n$-dimensional Frostman measure for $F_M$. That is, $\mu$ is a non-zero Radon measure supported on $F_M$
such that
$$\mu(B(x,r))\leq C\,r^n\qquad \mbox{for all $x\in\R^{n+1}$.}$$
Further, by renormalizing $\mu$, we can assume that $\|\mu\|=1$. Of course the constant $C$ above will depend on
$\HH^n_\infty(F_M)$, and the same may happen for all the constants $C$ to appear,  but this causes no problems. Notice that $\mu\ll\HH^n|_{F_M}\ll \omega^p$.
In fact, for any set $H\subset F_M$,
\begin{equation}\label{Frostman}
\mu(H)\leq C\,\HH^n_\infty(H)\leq C\,\HH^n(H)\leq C\,M\,\omega^p(H).
\end{equation}

\vv

\subsection{The bad cubes}\label{secbad}

Now we recall the definition of bad cubes from \cite{AMT}.
We say that $Q\in\DD$ is bad and we write $Q\in\bad$, if $Q\in\DD$ is a maximal cube satisfying one of the conditions below:
\begin{itemize}
\item[(a)] $\mu(Q)\leq \tau\,\omega^p(Q)$, where $\tau>0$ is a small parameter to be fixed below, or
\item[(b)]  $\omega^p(3B_Q)\geq A\,r(B_Q)^n$, where $A$ is some big constant to be fixed below.
\end{itemize}
The existence maximal cubes is guarantied by the fact that all the cubes from $\DD$ have side length uniformly bounded from
above (since $\DD_k$ is defined only for $k\geq0$).
If the condition (a) holds, we write $Q\in\LM$ (little measure $\mu$) and in the case (b), $Q\in\HD$ (high density).
On the other hand, if a cube $Q\in\DD$ is not contained in any cube from $\bad$, we say that $Q$ is good and we write
$Q\in\good$.

For technical reasons one needs to introduce a variant of the family $\DD^{db}$ of doubling cubes.
Given some constant $T\geq C_0$ (where $C_0$ is the constant in Lemma \ref{lemcubs}) to be fixed below,
we say that $Q\in\wt\DD^{db}$ if
$$
\omega^p(100B(Q))\leq T\,\omega^p(Q).
$$
We also set $\wt \DD^{db}_k=\wt\DD^{db}\cap \DD_k$ for $k\geq0$.
From \rf{eqdob*} and the fact that $T\geq C_0$, it is clear that $\DD^{db}\subset \wt\DD^{db}$.

It is shown then in Lemma 6.1 of \cite{AMT} that if $\tau$ is small enough and $A$ and $T$ big enough, then
$$\omega^p\biggl(F_M \cap \bigcup_{Q\in\wt\DD_0^{db}} Q  \setminus
 \bigcup_{Q\in\bad} Q
\biggr) >0,$$
where $\wt\DD_0^{db}$ stands for the family of cubes from the zero level of $\wt\DD^{db}$.

Notice that for the points $x\in F_M\setminus \bigcup_{Q\in\bad} Q$, from the condition (b) in the definition
of bad cubes, it follows that
$$\omega^p(B(x,r))\lesssim A\,r^n\qquad \mbox{for all $0<r\leq 1$.}$$
Trivially, the same estimate holds for $r\geq1$, since $\|\omega^p\|=1$. So we have
\begin{equation}\label{eqcc0}
\M^n\omega^p(x)\lesssim A\quad \mbox{ for $\omega^p$-a.e.\ $x\in F_M\setminus \bigcup_{Q\in\bad} Q$.}
\end{equation}

\vv

\subsection{The key lemma}

Folowinf the same arguments in  \cite{AMT}, it turns out that to prove Proposition \ref{propreg} it is enough to show the following.

\begin{lemma}[Key lemma]
Let $Q\in\good$ be contained in some cube from the family $\wt{\DD}_0^{db}$. Then we have
\begin{equation}\label{eqdk0}
\bigl|\RR_{r(B_Q)}\omega^p(x)\bigr| \leq  C(\delta,A,M,T,\tau,d(p))\quad \mbox{ for all $x\in B_Q$},
\end{equation}
where, to shorten notation, we wrote $d(p)= \dist(p,\partial\Omega)$.
\end{lemma}

The proof of this lemma in \cite{AMT} uses the porosity of $\partial\Omega$ in $E$. The proof below
does not, and instead uses some arguments of integration by parts which are not present in the analogous arguments from \cite{AMT}.

\begin{proof}We may assume that $r(B_Q)\ll d(p)=\dist(p,\partial\Omega)$, since otherwise \eqref{eqdk0} is trivial.
Further, by the same techniques as the ones from the proof of the Key Lemma 7.1 from \cite{AMT},
it is enough to show \rf{eqdk0} just for the cubes $Q\in\good\cap \wt{\DD}^{db}$.
Recall that, by definition, a cube
 $Q\in\wt\DD^{db}\cap\good$ satisfies in particular
 \begin{equation}\label{eqcond**}
\mu(Q)> \tau\,\omega^p(Q)\quad \mbox{and}\quad \omega^p(3B_Q)\leq T\,\omega^p(Q).
\end{equation}

Let $\vphi:\R^d\to[0,1]$ be a radial $\CC^\infty$ function  which vanishes on $B(0,1)$ and equal $1$ on $\R^d\setminus B(0,2)$,
and for $\ve>0$ and $z\in\R^{n+1}$ denote
$\vphi_\ve(z) = \vphi\left(\frac{z}\ve\right) $
and set
$$\wt\RR_\ve\omega^p(z) =\int K(z-y)\,\vphi_\ve(z-y)\,d\omega^p(y),$$
where $K(\cdot)$ is the kernel of the $n$-dimensional Riesz transform.

Let $\delta>0$ be the constant appearing in Lemma \ref{lembourgain} about Bourgain's estimate. Consider a ball $\wt B_Q$ centered at some point from $B_Q\cap\partial\Omega$ with $r(\wt B_Q)= \frac{\delta}{10} \,r(B_Q)$ such $\mu(\wt B_Q)\gtrsim \mu(B_Q)$, with the implicit constant depending
on $\delta$.
Note that, for every $x,z\in B_Q$, by standard Calderón-Zygmund estimates
$$\bigl|\wt\RR_{r(\wt B_Q)}\omega^p(x)- \RR_{r(B_Q)}\omega^p(z)\bigr|\leq C(\delta)\,\M^n_{r(\wt B_Q)}\omega^p(z),$$
and
$$\M^n_{r(\wt B_Q)}\omega^p(z)\leq C(\delta,A)\qquad \mbox{for all $z\in B_Q$,}$$
 since $Q$ being good implies that
$Q$ and all its ancestors are not from $\HD$.
Thus,
to prove \rf{eqdk0} it suffices to show that
\begin{equation}\label{eqdk00}
\bigl|\wt\RR_{r(\wt B_Q)}\omega^p(x)\bigr| \leq  C(\delta,A,M,T,\tau,d_p) \quad\mbox{ for the center $x$ of $\wt B_Q$,}
\end{equation}

To shorten notation, in the rest of the proof we will write $r=r(\wt B_Q)$, so that $\wt B_Q=B(x,r)$.
Recall that at the beginning of Section \ref{sect-Green} we show that under the current assumptions $G(p,\cdot)$ is in $W_0^{1,2}(\Omega)$ away from $p$.
 To prove \rf{eqdk00} we may therefore formally integrate by parts (see \cite{HM} for a justification of this):
\begin{align}\label{eqdef123}
&\wt\RR_r\omega^p(x)  =\int K(x-y)\,\vphi_r(x-y)\,d\omega^p(y) \\
&\qquad = -\int_\Omega \nabla_y G(y,p)\cdot \nabla_y\bigl[K(x-y) \,\vphi_r(x-y)\bigr]\,dm(y) + K(x-p)\,\vphi_r(x-p)\nonumber \\
&\qquad = -\int_\Omega \nabla_y G(y,p)\cdot \bigl[\nabla_y K(x-y)\,\vphi_r(x-y)\bigr]\,dm(y)\nonumber \\
&\qquad\qquad\ \ -
\int_\Omega \nabla_y G(y,p)\cdot \bigl[K(x-y)\,\nabla_y\vphi_r(x-y)\bigr]\,dm(y) +  K(x-p)\,\vphi_r(x-p)\nonumber \\
&\qquad=: -I - II + III.\nonumber
\end{align}

We will estimate the terms $I$, $II$, and $III$ separately. Notice first that
$$|III|\leq \frac1{|x-p|^n}\leq \frac1{d(p)^n}.$$
Concerning $II$, since $\supp\vphi_r(x-\cdot)\subset A(x,r,2r):=B(x,2r)\setminus B(x,r)$ and $\|\nabla \vphi_r\|_\infty\leq c/r$, we have
\begin{align*}
|II|&\leq\frac cr\int_{\Omega\cap A(x,r,2r)} |\nabla_y G(y,p)|\,|K(x-y)|\,dm(y)\\
&\leq c\,\,\frac1{r^{n+1}}\int_{\Omega\cap B(x,2r)} |\nabla_y G(y,p)|\,dm(y)
\\
&\leq c\,\,\left(\frac1{r^{n+1}}\int_{\Omega\cap B(x,2r)} |\nabla_y G(y,p)|^2\,dm(y)\right)^{1/2}
.
\end{align*}
Extend  $G(\cdot,p)$ by $0$  utside of $\overline{\Omega}$ and, abusing the notation, call this extension $G(\cdot,p)$.  Observe that $G(\cdot,p)$ is in $W^{1,2}(\re^{n+1})$ away from $p$ and has compact support, since we showed before that $G(\cdot,p)$ is in $W^{1,2}_0(\Omega)$ away from $p$. Note also that $G(\cdot,p)$ is subharmonic in $B(x,4\,r)$ since $r\ll d(p)$ and hece we can invoke Caccioppoli's inequality to conclude that
\begin{align*}
|II|\leq  c\left(\avint_{B(x,2r)} |\nabla_y G(y,p)|^2\,dm(y)\right)^{1/2}
\leq \frac c{r}\left(\avint_{B(x,3r)} |G(y,p)|^2\,dm(y)\right)^{1/2}.
\end{align*}

 To deal with the term $I$, we consider a small ball $B$ centered at $p$ with radius much smaller that $d(p)$ and we split the
domain of integration as $\Omega =  (\Omega\setminus B)\cup B$:
$$ I = \biggl(\int_{\Omega\setminus B}+ \int_B\biggr)\nabla_y G(y,p)\cdot \bigl[\nabla_y K(x-y)\,\vphi_r(x-y)\bigr]\,dm(y)=:
I_a + I_b.$$
The integral $I_b$ is easy to estimate. We just use that, for $y\in B$,
\begin{equation}\label{eq110}
|\nabla_yG(y,p)|\leq \frac c{|y-p|^{n}}\quad\mbox{ and }\quad  \bigl|\nabla_y K(x-y)\,\vphi_r(x-y)\bigr|\leq
\frac{c}{|x-y|^{n+1}}\leq \frac{c}{d(p)^{n+1}}.
\end{equation}
So we have
$$|I_b|\leq \frac{c}{d(p)^{n+1}} \int_B \frac1{|y-p|^{n}}\,dm(y)\leq c\,\frac{r(B)}{d(p)^{n+1}}\leq \frac{c}{d(p)^{n}}.$$

To estimate the integral $I_a$ we use the previous extension og $G$ and apply the divergence theorem:
\begin{align*} I_a &=
\int_{\re^{n+1}\setminus B} {\rm div}\bigl(G(\cdot,p)\, \bigl[\nabla K(x-\cdot)\,\vphi_r(x-\cdot)\bigr]\bigl)(y)\,dm(y) \\
&\qquad\quad-\int_{\re^{n+1}\setminus B} G(y,p)\, {\rm div}\bigl[\nabla K(x-\cdot)\,\vphi_r(x-\cdot)\bigr](y)\,dm(y)\\
& = \int_{\partial B} G(y,p)\, \bigl[\nabla_y K(x-y)\,\vphi_r(x-y)\bigr]\,\cdot N(y)\,d\sigma(y)\\
&\qquad\quad -\int_{\re^{n+1}\setminus B} G(y,p)\, {\rm div}\bigl[\nabla K(x-\cdot)\,\vphi_r(x-\cdot)\bigr](y)\,dm(y)
\\
&=: I_{a,1} + I_{a,2} ,
\end{align*}
where $N(\cdot)$ stands for the unit normal vector on $\partial B$ pointing to the interior of $B$ and $\sigma$  is the
surface measure on $\partial B$. Note that for  the second identity we have used the fact that the Green function belongs
to $W^{1,2}(\re^{n+1})$ away from $p$ and that it has compact support.
Using that, for $y\in\partial B$,
$$|G(y,p)|\leq \frac c{r(B)^{n-1}}$$
and the second estimate in \rf{eq110} it follows that
$$|I_{a,1}|\leq \frac c{r(B)^{n-1}\,d(p)^{n+1}}\,\sigma(\partial B) \leq c\,\frac{r(B)}{d(p)^{n+1}}
\le
\frac{c}{d(p)^{n}}.
$$
To deal with $I_{a,2}$, observe that $K(x-\cdot)$ is harmonic away from $x$, and thus
$${\rm div}\bigl[\nabla K(x-\cdot)\,\vphi_r(x-\cdot)\bigr](y) = \nabla_y K(x-y)\cdot \nabla_y\vphi_r(x-y).$$
Therefore, since $\supp (\nabla_y\vphi_r(x-\cdot))\subset A(x,r,2r)$, we have
\begin{align*}
|I_{a,2}|&\leq \int_{\Omega\cap A(x,r,2r)} |G(y,p)|\, \bigl|\nabla_y K(x-y)\bigl|\,\bigl|\nabla_y\vphi_r(x-y)\bigr|
\,dm(y)\\ &\leq \frac{c}{r}\,\,
\avint_{B(x,2r)} |G(y,p)|\,
\,dm(y)\\
&\leq \frac{c}{r}
\left(\avint_{B(x,2r)} |G(y,p)|^2\,
\,dm(y)\right)^{1/2}.
\end{align*}

If we gather the estimates obtained for the terms $I$, $II$, and $III$, we get
$$\bigl|\wt\RR_\ve\omega^p(x)\bigr|\lesssim \frac{1}{r}
\left(\avint_{B(x,3r)} |G(y,p)|^2\,
\,dm(y)\right)^{1/2} + \frac1{d(p)^n}.$$
Thus, to conclude the proof the key lemma it is enough to show that
\begin{equation}\label{eqsuf1}
\frac1r\,| G(y,p)|\lesssim 1\qquad\mbox{for all $y\in  B(x,3r)\cap\Omega$.}
\end{equation}
To prove this, observe that by Lemma \ref{l:w>G} (with $B= B(x,3r)$, $a=2\delta^{-1}$), for all $y\in B(x,3r)\cap\Omega$ ,
we have
$$\omega^{p}(B(x,6\delta^{-1}r))\gtrsim \inf_{z\in B(x,6r)\cap \Omega} \omega^{z}(B(x,6\delta^{-1}r))\, r^{n-1}\,|G(y,p)|.$$
On the other hand, by Lemma \ref{lembourgain}, for any $z\in B(x,6r)\cap\Omega$,
$$\omega^{z}(B(x,6\delta^{-1}r))\gtrsim \frac{\mu(B(x,6r))}{r^n}\geq \frac{\mu(\wt B_Q)}{r^n}.$$
Therefore we have
$$
\omega^{p}(B(x,6\delta^{-1}r))
\gtrsim
\frac{\mu(\wt B_Q)}{r^n}\, r^{n-1}\,|G(y,p)|,
$$
and thus
$$
\frac1r\,| G(y,p)|\lesssim \frac{\omega^{p}(B(x,6\delta^{-1}r))}{\mu(\wt B_Q)}.
$$
Now, recall that by construction $\mu(\wt B_Q)\gtrsim \mu(B_Q)\geq \mu(Q)$ and
$B(x,6\delta^{-1}r)=6\delta^{-1} \wt B_Q\subset 3B_Q$, since $r(\wt B_Q)=\frac\delta{10}r(B_Q)$, and so
we have
$$
\frac1r\,| G(y,p)|\lesssim
\frac{\omega^{p}(B(x,6\delta^{-1}r))}{\mu(\wt B_Q)}\lesssim \frac{\omega^{p}(3B_Q)}{\mu(Q)}\lesssim
\frac{\omega^{p}(Q)}{\mu(Q)}\leq C,$$
by \rf{eqcond**}. So \rf{eqsuf1} is proved and the proof of the Key lemma is complete.
\end{proof}
\vvv


\section{The proof of Theorem \ref{teo1} for a general domain  $\Omega  \subset \R^{n+1}$, $n\geq 2$}\label{secend}

First we need the following auxiliary result.

\begin{lemma}\label{lempot}
Let $\Omega$ be a proper domain in $\R^{n+1}$ and $p\in\Omega$.
Let $W\subset\partial\Omega$ be the set of Wiener irregular points for $\Omega$. Then there exists a finite measure $\mu$ such that
$U^\mu(x)=\infty$ for all $x\in W$ and $U_\mu(p)\leq 1$.
\end{lemma}

\begin{proof}
For $x\in\partial \Omega$, denote $$S(x) = \int_0^1  \frac{{\rm Cap}(A(x,r,2r)\cap \Omega^c)}{r^{n-1}}\,\frac{dr}{r},$$
so that $x$ is regular if and only if $S(x)=\infty$.
Since $S$ is lower semicontinuous, for all $\lambda>0$ the set $\{x\in\R^{n+1}:S(x)>\lambda\}$ is open and thus the set of Wiener regular point is a $G_\delta$ set (relative to $\partial\Omega$). Thus the set $W$ of the irregular points from
$\partial\Omega$ is an
$F_\sigma$ set. Thus we can write
$$W=\bigcup_{j\geq 1} K_i,$$
where each $K_i$ is a compact subset of $\partial \Omega$.

By Kellog's Lemma \cite[p.232]{Landkof}, we know that $\CCap(W)=0$ and thus $\CCap(K_i)=0$ for all $i$. Then, by Theorem 3.1 of \cite{Landkof}, for
each $i$ there exists a finite measure $\mu_i$ such that $U^{\mu_i}(x)=\infty$ for all $x\in K_i$ and
$U^{\mu_i}(x)<\infty$ for all $x\not \in K_i$. So the measure
$$
\mu
=
\sum_{i\geq 1}
\frac{1}{2^{i}\, \max\{U^{\mu_i}(p),\|\mu_i\|\}}\,\mu_i$$
satisfies the requirements of the lemma.
\end{proof}

\vvv

We are now ready to prove Theorem \ref{teo1} for bounded domains $\re^{n+1}$, $n\ge 2$.
Let $p\in\Omega$.
Consider $E\subset\partial\Omega$ with $0<\HH^n(E)<\infty$ such that the harmonic measure $\omega^p|_E$ is absolutely continuous with respect to $\HH^n|_{E}$.
To prove the $n$-rectifiability of $\omega^p|_E$ it suffices to show that any
subset $F\subset E$ with $\omega^p(F)>0$ contains some $n$-rectifiable
subset $G$ with positive $\HH^n$ measure (hence the totally unrectifiable part of $E$ will have $w^p$-measure $0$).
To this end, we consider the measure $\mu$ in Lemma \ref{lempot}.
For a big $\lambda>0$ to be fixed below, we take the open set
$$V_\lambda = \{x\in\R^{n+1}:U^\mu(x)>\lambda\}.$$
Note that the set of irregular points $W$ from $\pom$ is contained in $V_\lambda\cap\pom$, for any $\lambda>0$.

Now we will construct an auxiliary domain $\wt\Omega$ (to which we will later apply Proposition \ref{propreg}) as follows. For each $x\in W$, consider a radius  $0<r_x\le\min\{1,d(p)/2\}$ such that
the closed ball $\bar B(x,r_x)$ is contained in $V_\lambda$, and we apply the Besicovitch covering lemma to get
a family of closed balls $B_i$, $i\in I$, centered at points from $W$, which cover $W$ and have bounded overlap.
Then we define
$$\wt \Omega = \Omega\setminus \bigcup_{i\in I} B_i.$$

We will show now that $\wt \Omega$ is open. Indeed, we claim that
\begin{equation}\label{eqcla**}
\overline{\bigcup_{i\in I} B_i} \setminus \bigcup_{i\in I} B_i \subset\partial\Omega.
\end{equation}
This inclusion implies that
$$ \Omega\setminus \overline{\bigcup_{i\in I} B_i} =
\Omega\setminus \left[\left(\overline{\bigcup_{i\in I} B_i} \setminus \bigcup_{i\in I} B_i\right)\cup
\bigcup_{i\in I} B_i\right] = \Omega\setminus \bigcup_{i\in I} B_i = \wt\Omega,$$
and thus ensures that $\wt\Omega$ is open.

To show our claim \rf{eqcla**} consider $x\in\overline{\bigcup_{i\in I} B_i} \setminus \bigcup_{i\in I} B_i$
and recall that, by construction each ball $B_i$ is closed. Then $x$ must be the limit of a sequence of points
belonging to infinitely many different balls $B_{i_k}$, $i_k\in I$. It turns out that then we have $r(B_{i_k})\to 0$.
This is a straightforward consequence of the fact that
any family of balls $B_j$, $j\in J\subset I$, such that
$\dist(B_j,x)\leq 1$ and $0<\ve \leq r(B_j)\leq 1$ must be finite, by the finite overlap of the family $\{B_i\}_{i\in I}$.
The fact
that  $r(B_{i_k})\to0$ implies that $x\in\partial \Omega$, since the balls $B_{i,k}$ are centered in $\partial\Omega$.

From \rf{eqcla**} we also deduce that
\begin{equation}\label{eqcla**2}
\partial\wt\Omega\subset \biggl(\partial\Omega
\setminus \bigcup_{i\in I} B_i\biggr) \cup\bigcup_{i\in I}\partial B_i.
\end{equation}
To see this, write
\begin{align*}
\partial\wt\Omega = \partial\biggl(\Omega\setminus \bigcup_{i\in I} B_i\biggr) & \subset
\partial\Omega \cup   \overline{\bigcup_{i\in I} B_i} \\
& = \partial\Omega \cup
\biggl(\,\overline{\bigcup_{i\in I} B_i} \setminus \bigcup_{i\in I} B_i\biggr)
\cup \bigcup_{i\in I} B_i \\
& = \partial\Omega
\cup \bigcup_{i\in I} B_i = \biggl(\partial\Omega
\setminus \bigcup_{i\in I} B_i\biggr)
\cup \bigcup_{i\in I} B_i
.
\end{align*}
On the other hand, by construction the interior of each ball $B_i$ lies in the exterior of $\wt\Omega$, and thus
$$\partial\wt\Omega = \partial\wt\Omega\setminus {\rm ext}(\wt\Omega)
 \subset \biggl[\biggl(\partial\Omega
\setminus \bigcup_{i\in I} B_i\biggr) \cup\bigcup_{i\in I}B_i\biggr] \setminus {\rm ext}(\wt\Omega) \subset \biggl(\partial\Omega
\setminus \bigcup_{i\in I} B_i\biggr) \cup\bigcup_{i\in I}\partial B_i,$$
which proves \rf{eqcla**2}.

We wish to show now that, if $\lambda$ has been chosen big enough, then there exists some subset
$\wt F\subset F\cap\partial\wt\Omega$ with positive harmonic measure $\wt\omega^p$  (this is the harmonic measure for $\wt\Omega$ with pole at $p\in\wt\Omega$, that $p\in\wt\Omega$ follows from the fact that $r(B_i)\le d(p)/2$)  such that $\wt\omega^p|_{\wt F}\ll \HH^n|_{\wt F}$. Denote
$$\wt B= \bigcup_{i\in I}\partial B_i\quad\text{ and }\quad \wt G= \partial\wt\Omega\setminus \wt B.$$
Note that \rf{eqcla**2} tells us that $\wt G\subset \partial \Omega\cap \partial\wt\Omega$.
By  a formal application of the maximum principle and the construction of $\wt\Omega$ we have
\begin{equation}\label{max-pple-1}
\wt\omega^p(\wt G\setminus F)\leq \omega^p(\wt G\setminus F).
\end{equation}
We would like to emphasize that our use of the maximum principle strongly uses the construction of harmonic measure solutions using Perron's method. We are working in a regime where the
Wiener test may fail, and the involve solutions are not
Perron solutions for the same domain, nor are they continuous on the closures of
the respective domains under consideration. Hence, classical maximum principle does not apply. We shall give a rigorous justification at the end of the proof, see \ref{eq:just-max-pple}.

On the other hand, observe that $\wt B\subset V_\lambda$. Then we consider the function $f(x) = \frac1\lambda \,U^\mu(x)$, which
is superharmonic in $\R^{n+1}$, with $f(x)>1$ for all $x\in V_\lambda$ (and thus for all $x\in\wt B$), and $f(p)\leq 1/\lambda$.  By the maximum principle (here it is just the Perron method), then we deduce that
$$\wt\omega^p(\wt B) \leq f(p)\leq \frac1\lambda.$$
Hence, choosing $\lambda =2/\omega^p(F)$,
\begin{align*}
\wt\omega^p(F\cap\partial\wt\Omega) & \geq \wt\omega^p(\partial\wt\Omega) - \wt\omega^p(\wt G\setminus F) - \wt\omega^p(\wt B)\\
& \geq \omega^p(\partial\Omega) - \omega^p(\wt G\setminus F) - \wt\omega^p(\wt B) \\&\geq \omega^p(F) - \wt\omega^p(\wt B)
\\&\geq \frac12\,\omega^p(F)>0.
\end{align*}
Then we take $\wt F:=F\cap\partial\wt\Omega$, so that $\wt\omega^p(\wt F)>0$. Further, by the maximum principle (which again requires some justification, see \eqref{eq:just-max-pple}) and the fact that $
\omega^p|_{F}\ll \HH^n|_{F}$, we infer that
$$\wt\omega^p|_{\wt F}\ll \omega^p|_{\wt F}\ll \HH^n|_{\wt F}.$$

We intend to apply Proposition \ref{propreg} to show that $\wt\omega^p|_{\wt F}$ is $n$-rectifiable.
To this end, it remains to check that
 $\wt\Omega$ is Wiener regular. That is, all the points $x\in\partial\wt\Omega$ are Wiener regular for
$\wt\Omega$. We have to show that
$$\int_0^1 \frac{{\rm Cap}(A(x,r,2r)\cap \wt\Omega^c)}{r^{n-1}}\,\frac{dr}{r} = \infty$$
for all $x\in\partial\wt\Omega$. By \rf{eqcla**2} we know that either $x\in\left(\partial\Omega
\setminus \bigcup_{i\in I} B_i\right)
$ or $x\in\partial B_i$ for some
$i\in I$.
In the latter case we have
$$\int_0^1 \frac{{\rm Cap}(A(x,r,2r)\cap \wt\Omega^c)}{r^{n-1}}\,\frac{dr}{r} \geq
\int_0^1 \frac{{\rm Cap}(A(x,r,2r)\cap B_i)}{r^{n-1}}\,\frac{dr}{r}=\infty,$$
since the  complement of any ball $B_i$ is Wiener regular.

If $x\in\partial\Omega \setminus \bigcup_{i\in B_i} B_i$, then we know that $x$ is Wiener regular for $\Omega$, since
$W\subset\bigcup_{i\in I} B_i$. Thus, using just that $\wt\Omega^c\supset \Omega^c$, we obtain
$$\int_0^1 \frac{{\rm Cap}(A(x,r,2r)\cap \wt\Omega^c)}{r^{n-1}}\,\frac{dr}{r} \geq
\int_0^1 \frac{{\rm Cap}(A(x,r,2r)\cap \Omega^c)}{r^{n-1}}\,\frac{dr}{r} =\infty.$$
So the proof that $\wt \Omega$ is Wiener regular is concluded.

Now we can apply Proposition \ref{propreg} to deduce that $\tilde{\omega}^p|_{\wt F}$
is rectifiable. In other words, there exists
an $n$-rectifiable subset $G\subset\wt F$ and $g\in L^1(\HH^n|_G)$ such that
$$\wt\omega^p|_{\wt F} = g\,\HH^n|_G.$$
The fact that $\wt\omega^p(\wt F)>0$ ensures that $\HH^n(G)>0$, as wished.

To conclude this proof we need to justify the use of maximum principle which is based on Perron's construction of harmonic measure. We are going to show that
\begin{equation}
\label{eq:just-max-pple}
\wt\omega^p(\mathcal{O})\leq \omega^p(\mathcal{O}),
\qquad\mbox{for every Borel set }\mathcal{O}\subset\pom\cap \partial\wt\Omega.
\end{equation}
Set $u(x):=\omega^x(\mathcal{O})$, $x\in\Omega$, which is the harmonic measure solution associated with the boundary data $\chi_{\mathcal{O}}\in L^\infty(\pom)$ via Perron's method, see for instance \cite[Chapter 2]{GT}. We pick $\varphi$ an arbitrary superfunction relative to $\chi_{\mathcal{O}}$ for $\Omega$, that is, $\varphi\in C(\overline{\Omega})$, $\varphi$ is superharmonic  in $\Omega$, and $\varphi\ge\chi_{\mathcal{O}}$ in $\pom$. Let us recall that $u$ is precisely the infimum of all these superfunctions.
Note that $\phi\equiv 0$ is a subfunction relative to $\chi_{\mathcal{O}}$, since it is clearly harmonic, continuous everywhere and $\phi\le \chi_{\mathcal{O}}$ on $\pom$. Hence, by the maximum principle for subharmonic and superharmonic functions
that are continuous up to the boundary, we conclude that $0\le \varphi$ in $\overline{\Omega}$.

Let us check that $\varphi\ge\chi_{\mathcal{O}}$ in $\partial\wt\Omega=(\partial\wt\Omega\cap\Omega)\cup(\partial\wt\Omega\cap\pom)$. Our choice of $\varphi$ guarantees that $\varphi\ge \chi_{\mathcal{O}}$ in $ \partial\wt\Omega\cap\Omega$. On the other hand, if $x\in \partial\wt\Omega\cap\Omega$ we have $\varphi(x)\ge 0=1_{\mathcal{O}}(x)$. On the other hand, clearly
 $\varphi\in C(\overline{\wt \Omega})$ is superharmonic in $\wt \Omega$. We have then show that $\varphi$ is a superfunction relative to $\chi_{\mathcal{O}}$ for $\wt\Omega$ and hence Perron's method in $\wt\Omega$ gives that
$\wt\omega^p(\mathcal{O})\le \varphi$. We now take the infimum over all such $\varphi$ to conclude by  Perron's method in $\Omega$ that $\wt\omega^x(\mathcal{O})\leq \omega^x(\mathcal{O})$ holds for every $x\in\wt\Omega$. This completes our proof.

\section{Proof of Theorem \ref{teo1} in the planar case $n+1=2$}

To start, as in the higher dimensional case, we immediately reduce Theorem \ref{teo1} to the case of bounded domains.

\subsection{Logarithmic capacity, Wiener regular points and Green function}

The logarithmic potential of a measure $\mu$ in $\R^{2}$ is defined as
$$
U^\mu(x)
:= \int \log \frac1{|x-y|}\,d\mu(y).
$$
The Wiener capacity of a Borel compact set $A\subset \R^{n+1}$ is then defined by
$$
\CCap(A) = \sup\{\mu(A):\,U^\mu(x)\leq 1,\ \,\forall x\in\R^{n+1}\},
$$
(see, e.g., \cite{Landkof}, p. 168, in combination with \cite{Landkof}, Theorem 2.8).

Given $\Omega\subsetneq\re^{n+1}$. We say that a point $x\in\partial\Omega$ is Wiener regular for $\Omega$ (or just regular) if
$$
\int_0^1 {\rm Cap}(A(x,r, 2r)\cap \Omega^c)\,\log \frac 1r\,\frac{dr}{r} = \infty.
$$
If $x$ is not regular, we say that it is irregular. We say that $\Omega$ is Wiener regular if every $x\in\pom$ is Wiener regular. (See \cite{Landkof}, Theorem 5.6).

 In this case the Green function is defined as follows. Much as before we set $G$ as in \eqref{green} with $\mathcal{E}$ replaced by $\frac1{2\,\pi} \log \frac1{|x|}$ which will act as a fundamental solution:
\begin{equation}
\label{eq:Green-2}
G(x,y)
=
\frac1{2\,\pi} \log \frac1{|x-y|}-
\int_{\partial\Omega} \frac1{2\,\pi} \log \frac1{|x-z|}d\omega^y(z).
\end{equation}
Now
we can,  \textit{mutatis mutandis}, repeat the argument carried out in Section \ref{sect-Green} to conclude much as before that $G(x,\cdot)\in W_0^{1,2}(\Omega)$ away from $x$ when the domain is Wiener regular.

At this point we can formulate Proposition~\ref{propreg} identically to the original statement, but with $n=1$ and, respectively, with the definition of Wiener regularity as above. Let us discuss the modifications in its proof compared to the higher dimensional case.

\subsection{Proof of the Key Lemma in the planar case $n+1=2$}
We recall that in this section the domain is assumed to be Wiener regular. We note  that the arguments to prove Lemma \ref{l:w>G} fail in the planar case. Therefore this cannot be
applied to prove the Key Lemma and some changes are required. We follow the same scheme and notation and highlight the important modifications.

We claim that for any constant $\alpha\in \R$,
\begin{equation}\label{eq1**}
\bigl|\wt\RR_r\omega^p(x)\bigr|\lesssim \frac{1}{r}
\left(\avint_{\Omega\cap B(x,3r)} |G(y,p) - \alpha|^2\,
\,dm(y)\right)^{1/2} + \frac1{d(p)}.
\end{equation}
To check this, recall that in the proof of the Key Lemma for $n\geq2$ we showed that
$$|\wt\RR_r\omega^p(x)| \leq |I| + |II| + |III|,$$
with the terms $I$, $II$ and $III$ being defined in \rf{eqdef123}. Note that we the formal integration by parts argument
can be done in a more or less standard way following for instance the ideas in \cite{HM}
with the appropriate changes. Details are left to the interested reader.

Much as before we can show that
$$
|III|\lesssim\frac1{d(p)^n}
$$
(now with $n=1$)
and also that
$$
|II| \leq  c\left(\avint_{B(x,2r)} |\nabla_y G(y,p)|^2\,dm(y)\right)^{1/2},
$$
which by Caccioppoli's inequaltity gives
$$
|II|
\leq \frac c{r}\left(\avint_{B(x,3r)} |G(y,p)-\alpha|^2\,dm(y)\right)^{1/2}
$$
for any $\alpha\in\R$. Again, we extend the Green function by $0$ outside of $\Omega$,
Concerning the term $I$, we have as before
$$|I|\leq \frac c{d(p)^n} + |I_{a,2}|,$$
with
\begin{align}\label{eqjj2}
I_{a,2} :&=
\int_{\R^{n+1}\setminus B} G(y,p)\, {\rm div}\bigl[\nabla K(x-\cdot)\,\vphi_r(x-\cdot)\bigr](y)\,dm(y)\\
& =
\int_{\R^{n+1}\setminus B} (G(y,p)-\alpha)\, {\rm div}\bigl[\nabla K(x-\cdot)\,\vphi_r(x-\cdot)\bigr](y)\,dm(y) \nonumber\\
&\quad + \alpha\int_{\R^{n+1}\setminus B} {\rm div}\bigl[\nabla K(x-\cdot)\,\vphi_r(x-\cdot)\bigr](y)\,dm(y).\nonumber
\end{align}
To estimate the last integral on the right hand side, observe first that the integrand is compactly supported because
$${\rm div}\bigl[\nabla K(x-\cdot)\,\vphi_r(x-\cdot)\bigr](y) = \nabla_y K(x-y)\cdot \nabla_y\vphi_r(x-y).$$
Then, for any big $R>0$ so that $B(x,R)$ contains $B$, by the divergence theorem the last integral on the right hand
side of \rf{eqjj2} equals
\begin{align*}
J & :=
\int_{B(x,R)\setminus B} {\rm div}\bigl[\nabla K(x-\cdot)\,\vphi_r(x-\cdot)\bigr](y)\,dm(y)\\
& = \int_{\partial B(x,R)}\nabla_y K(x-y)\,\cdot N(y)\,d\sigma(y) +
\int_{\partial B}\nabla_y K(x-y)\cdot N(y)\,d\sigma(y),
\end{align*}
where  we took into account that $\vphi_r(x-\cdot)$ is identically $1$ on $\partial B$ and $\partial B(x,R)$. In the previous expression $N(\cdot)$ stands for the unit normal pointing to the exterior in the first integral and pointing to the interior in the second integral.
It is easy to check that the first integral on the right hand side is bounded above by $C/R$ and the second one
by $C\,r(B)/d(p)^2$. So letting $R\to\infty$ we obtain
$$|J|\lesssim \frac{r(B)}{d(p)^2}.$$
Hence we deduce that
$$|I_{a,2}| \leq
\int_{\R^{n+1}\setminus B} |G(y,p) - \alpha|\, {\rm div}\bigl[\nabla K(x-\cdot)\,\vphi_r(x-\cdot)\bigr](y)\,dm(y)
+ \frac{C\,|\alpha|\,r(B)}{d(p)^2}.$$
To estimate the first integral on the right hand side we proceed as with the analogous integral with $\alpha=0$ in the
proof of the Key Lemma in the case $n>1$:
since $\supp (\nabla_y\vphi_r(x-\cdot))\subset A(x,r,2r)$, we get
\begin{align*}
\int_{\Omega\cap A(x,r,2r)} |G(y,p)-\alpha|\, \bigl|\nabla_y K(x-y)\bigl|\,&\bigl|\nabla_y\vphi_r(x-y)\bigr|
\,dm(y)\\ &\leq \frac{c}{r}\,\,
\avint_{B(x,2r)} |G(y,p)-\alpha|\,
\,dm(y)\\
&\leq \frac{c}{r}
\left(\avint_{ B(x,2r)} |G(y,p)-\alpha|^2\,
\,dm(y)\right)^{1/2}.
\end{align*}
Gathering all the estimates for the terms $I$, $II$ and $III$, we obtain
$$\bigl|\wt\RR_\ve\omega^p(x)\bigr|\lesssim \frac{1}{r}
\left(\avint_{B(x,3r)} |G(y,p) - \alpha|^2\,
\,dm(y)\right)^{1/2} + \frac1{d(p)} + \frac{|\alpha|\,r(B)}{d(p)^2}.$$
Since the estimates above are uniform on $r(B)$ (for $r(B)$ small enough), letting $r(B)\to0$, our claim
\rf{eq1**} follows.

Choosing $\alpha=G(z,p)$ with $z\in 3B$ in \rf{eq1**}, averaging with respect Lebesgue measure for such $z$'s, and applying
 applying H\"older's inequality, we get
$$\bigl|\wt\RR_\ve\omega^p(x)\bigr|\lesssim \frac1{r^3}\left(\iint_{B(x,3r)\times B(x,3r)} |G(y,p) - G(z,p)|^2\,dm(y)\,dm(z)\right)^{1/2} + \frac1{d(p)},$$
where we understand that $G(z,p)=0$ for $z\not\in\Omega$.
Now for $y,z\in B(x,3r)$ and $p$ far away we write (cf. \eqref{eq:Green-2})
\begin{align*}
2\,\pi\,(G(y,p) - G(z,p))
& =
\log\frac{|z-p|}{|y-p|} - \int_{\pom} \log\frac{|z-\xi|}{|y-\xi|} \,d\omega^p(\xi) \\
& =
\left(\log\frac{|z-p|}{|y-p|} - \int_{\partial\Omega} \phi\left(\frac{\xi-x}{r}\right)\,\log\frac{|z-\xi|}{|y-\xi|} \,d\omega^p(\xi)\right) \\
&\quad- \int_{\partial\Omega} \left(1-\phi\left(\frac{\xi-x}{r}\right)\right)\log\frac{|z-\xi|}{|y-\xi|} \,d\omega^p(\xi)\\
&= A_{y,z}  + B_{y,z},
\end{align*}
where $\phi$ is a radial smooth function such that $\phi\equiv 0$ in $B(0,4)$ and $\phi\equiv 1$ in $B(0,5)$.
Notice that the above identities also hold if $y,z\not\in\Omega$.
Let us observe that
$$\frac{|z-p|}{|y-p|}\approx 1$$
and
$$
\frac{|z-\xi|}{|y-\xi|}\approx 1\quad\mbox{ for $\xi\not \in B(x,4r)$,}
$$
We claim that
\begin{equation}\label{eqlem33}
|A_{y,z}|\lesssim \frac{\omega^p(B(x,6\delta^{-1}r))}{\inf_{z\in B(x,6r)\cap \Omega} \omega^z(B(x,6\delta^{-1}r))}.
\end{equation}
We defer the details till the end of the proof.
Using Bourgain's estimate (cf. Lemma \ref{lembourgain}) we get
$$\inf_{z\in B(x,6r)\cap \Omega}\omega^{z}(B(x,6\delta^{-1}r))\gtrsim \frac{\mu(B(x,6r))}{r}\geq \frac{\mu(\wt B_Q)}{r}.$$
and thus
$$\frac{|A_{y,z}|}r\lesssim \frac{\omega^p(B(x,6\delta^{-1}r))}{\mu(\wt B_Q)}\lesssim
\frac{\omega^{p}(Q)}{\mu(Q)},$$
by the doubling properties of $Q$ (for $\omega^p$) and the choice of $\wt B_Q$.

To deal with the term $B_{y,z}$ first we use H\"older's inequality:
\begin{align*}
|B_{y,z}|^2&\leq \omega^p(B(x,5r))\int_{B(x,5r)} \left|\log\frac{|z-\xi|}{|y-\xi|}\right|^2 \,d\omega^p(\xi)\\
& \lesssim
\omega^p(B(x,5r))\int_{ B(x,5r)} \left(\left|\log\frac{r}{|y-\xi|}\right|^2 +  \left|\log\frac{r}{|z-\xi|}\right|^2\right) \,d\omega^p(\xi).
\end{align*}
Thus
\begin{align*}
\iint_{B(x,3r)\times B(x,3r)} &|B_{y,z}|^2\,dm(y)\,dm(z) \\
&\lesssim
\omega^p(B(x,5r))\,r^2\int_{B(x,3r)}\int_{ B(x,3r)} \left|\log\frac{r}{|y-\xi|}\right|^2 \,d\omega^p(\xi)\,dm(y).
\end{align*}
Notice that for all $\xi\in B(x,5r)$,
$$\int_{B(x,3r)}\left|\log\frac{r}{|y-\xi|}\right|^2 \,dm(y)\lesssim r^2.$$
So by Fubini we obtain
$$\iint_{B(x,3r)\times B(x,3r)} |B_{y,z}|^2\,dm(y)\,dm(z) \lesssim \omega^p(B(x,4r))^2\,r^4.$$
That is,
$$
\frac1{r^3}\left(\iint_{B(x,3r)\times B(x,3r)} |B_{y,z}|^2\,dm(y)\,dm(z)\right)^{1/2}\lesssim \frac{\omega^p(B(x,5r))}r
.$$
Together with the bound for the term $A_{y,z}$, this gives

$$\bigl|\wt\RR_\ve\omega^p(x)\bigr|\lesssim \frac{\omega^{p}(Q)}{\mu(Q)} +\frac{\omega^p(B(x,5r))}r + \frac1{d(p)}\lesssim 1,$$
since $\M^1\omega^p(x)\lesssim1$ by \eqref{eqcc0}.
\vv

It remains now to show \rf{eqlem33}. The argument uses the ideas in Lemma \ref{l:w>G} with some modifications.
Recall that
\begin{align*}
A_{y,z}
&=
A_{y,z}(p) =
\log\frac{|z-p|}{|y-p|} - \int_{\partial\Omega} \phi\left(\frac{\xi-x}{r}\right)\,\log\frac{|z-\xi|}{|y-\xi|} \,d\omega^p(\xi)
\\
&=:
\log\frac{|z-p|}{|y-p|}-v_{x,y,z}(p)
\end{align*}
where $y,z\in B(x,3r)$ and $p$ is far away.
The two functions
$$
q\longmapsto A_{y,z}(q)\qquad\text{ and }\qquad q\longmapsto \frac{c\,\omega^q(B(x,6\delta^{-1}r))}{\inf_{z\in B(x,6r)\cap \Omega} \omega_\Omega^{z}(B(x,6\delta^{-1}r))}$$
 are harmonic in $\Omega\setminus B(x,6r)$. Note that for all $q\in\partial B(x,6r)$
we clearly have
$$|A_{y,z}(q)|\leq c\leq \frac{c\,\omega^q(B(x,6\delta^{-1}r))}{\inf_{z\in B(x,6r)\cap \Omega} \omega_\Omega^{z}(B(x,6\delta^{-1}r))}.$$
Note also that $v_{x,y,z}$ is a harmomic function associated with a smooth boundary data, and, in particular, the fact that domain is Wiener regularity implies that $v_{x,y,z}\in C(\overline{\Omega})$. Thus $A_{y,z}(x)=0$ for every $x\in\pom\setminus B(x,5)$. Hence we can apply maximum principle (this will require a justification completely analogous to that at the end of the proof of Lemma \ref{l:w>G})
and obtain as desired  \rf{eqlem33}.

\fiproof

\subsection{End of the proof Theorem~\ref{teo1} in the planar case $n+1=2$}
This section discusses modifications in the arguments of Section~\ref{secend} pertinent to the planar case.

First of all, Lemma~\ref{lempot} continues to hold for $n=1$ with the logarithmic potential $U^\mu$ defined as above. In its proof, one has to take
$$S(x) = \int_0^1  {\rm Cap}(A(x,r,2r)\cap \Omega^c) \,\log \frac 1r\,\,\frac{dr}{r}.$$
The Kellogg's Lemma in the planar case also can be found in \cite{Landkof}, p. 232 (note that the sets of zero logarithmic capacity and sets of zero Wiener capacity are identical, see, e.g., \cite{Landkof}, p. 167). Theorem 3.1 of \cite{Landkof} also extends to the context of logarithmic potential (see Remark on p. 182 of \cite{Landkof}), and the rest of the argument of Lemma~\ref{lempot} is the same as in the higher dimensional case.

At this stage, the argument of Theorem~\ref{teo1} follows verbatim, with the only addition of a logarithmic factor $\log \frac 1r$ in the integrals of capacitory expressions in the end of the proof.

\vvv

\end{document}